\documentclass[a4paper,reqno]{amsart}
\usepackage{amsmath,amsthm,amssymb,amsfonts,amsbsy}
\usepackage{enumitem,color,graphicx}
\usepackage[all,pdf]{xy}

\usepackage[backref=page,hyperindex=true,CJKbookmarks=true,
colorlinks,linkcolor=blue,anchorcolor=red,citecolor=cyan]{hyperref}

\newtheorem{thm}{Theorem}[section]
\newtheorem{rmk}[thm]{Remark}
\newtheorem{defn}[thm]{Definition}
\newtheorem{prop}[thm]{Proposition}
\newtheorem{lem}[thm]{Lemma}

\theoremstyle{plain}
\newtheorem{maintheorem}{Theorem}

\newtheorem{main-cor}[maintheorem]{Corollary}

\numberwithin{equation}{section}

\renewcommand*{\backref}[1]{}
\renewcommand*{\backrefalt}[4]{%
    \ifcase #1 (Not cited.)%
    \or        (Cited on page~#2.)%
    \else      (Cited on pages~#2.)%
    \fi}

\address{Department of Mathematics, 
            Southern University of Science and Technology, 
            Shenzhen, Guangdong 518055, China}
\email{xiamy@sustech.edu.cn}

\title[Kan-type transitivity]{Topological transitivity of Kan-type partially hyperbolic diffeomorphisms}
\author{Mingyang Xia}

\date{\today}


\begin{document}

\begin{abstract}
	We present the topological transitivity of
	a class of diffeomorphisms on the thickened torus,
	including the partially hyperbolic example introduced by Ittai Kan in 1994,
	which is well known for the first systems with the intermingled basins phenomenon.
\end{abstract}

\maketitle

\section{Introduction}

In 1994, Kan \cite{Kan94} constructed two important partially hyperbolic systems,
which inspired a lot of studies from the statistical and topological viewpoints.

Precisely speaking, for the non-invertible case,
Kan constructed a $C^\infty$ skew-product map $f:M\to M$
on the 2-dimensional cylinder $M=\mathbb{S}^1\times[0,1]$ defined by
\begin{equation}\label{Ex-endo}
	f(x,t)=\left(3x, t+\frac{t(1-t)}{32}\cos(2\pi x)\right).
\end{equation}
For the invertible case,
Kan constructed a $C^\infty$ skew-product diffeomorphism
on the thickened torus $M=\mathbb{T}^2\times[0,1]$ defined by
\begin{equation}\label{Ex-diffeo}
	f(x,y,t)=\left(3x+y, 2x+y, t+\frac{t(1-t)}{32}\cos(2\pi x)\right).
\end{equation}

From the statistical viewpoint, it is well known that
Kan's constructions admit two physical measures whose basins are intermingled.
Recall that for a $C^2$ map $f$ on a compact manifold $M$,
the basin $\mathcal{B}(\mu)$ of an $f$-invariant Borel probability measure $\mu$ is defined by
\begin{displaymath}
	\mathcal{B}(\mu)= \left\{ x\in M :
	\lim\limits_{n\to\infty}
	\frac{1}{n}  \sum_{k=0}^{n-1}\phi(f^k(x)) =\int_{M}\phi\ d\mu
	~~{\rm for}~{\rm every}~ \phi \in C^0(M,\mathbb{R})   \right\}.
\end{displaymath}
Then, $\mu$ is called a \emph{physical measure} (or \emph{Sinai-Ruelle-Bowen} measure)
if the basin $\mathcal{B}(\mu)$ has positive Lebesgue measure.
The existence and finiteness of physical measures
are crucial in the study of statistical behavior of dynamical systems.
As is shown by Sinai, Ruelle and Bowen \cite{Si72,Ru76,BR75} that
uniformly hyperbolic systems only have finitely many physical measures
while the union of their basins has full Lebesgue measure in the ambient manifold.
See  \cite{CYZ20,HYY20,CMY22} for recent advances with weak hyperbolicity.
Moreover, two physical measures $\mu_1$ and $\mu_2$ have
\emph{intermingled basins}
if for any non-empty open set $U\subseteq M$,
\begin{displaymath}
	Leb(\mathcal{B}(\mu_1)\cap U)>0 \quad\text{and}\quad Leb(\mathcal{B}(\mu_2)\cap U)>0 .
\end{displaymath}
Kan's examples continue to be a source of interesting research in dynamical systems,
which have been extensively studied, especially in the aspect of intermingled basins,
see \cite{MW05,DVY16,BP18,UV18} for recent developments.

From the perspective of topology,
it has been shown that a typical class of Kan's non-invertible systems is indecomposed
by appending two conditions (see \cite[Proposition 11.2]{BDV05}).
In fact, it is proved this kind of cylinder endomorphisms is \emph{topologically transitive},
that is, for any two nonempty open sets $U,V\subseteq M$,
there is a positive integer $m$ such that $f^m(U)\cap V\neq\varnothing.$
Recently, Gan and Shi \cite{GS19} show that Kan-type cylinder endomorphisms,
including  Example \eqref{Ex-endo},
are robustly topologically mixing within the $C^2$ boundary preserving maps.

These facts indeed indicate an interesting difference
between the measure theory and the topology:
these constructions admit two intermingled physical measures
but are topologically indecomposed.
This makes Kan's examples more important to some extent.
The contrast of measurable and topological properties
is an interesting subject for the study of dynamics.
Very recently, there are some beautiful results
focusing on topological transitivity of skew-products (see \cite{Ok17,CO21})
and measures of maximal entropy for some general systems
related to the Kan's endomorphism (see \cite{NBRV21,RT22}).

As for the situation of Kan's invertible systems,
there are few studies related to Kan's constructions.
In contrast to the robust manner of the intermingled basins (see \cite{IKS08}),
in 2018, Ures and V\'{a}squez \cite{UV18} establish the non-robust intermingled basins phenomenon
on $\mathbb{T}^3$, providing the constructions are not accessible.
Soon after, appending extra constructions,
a family of topologically transitive diffeomorphisms on $\mathbb{T}^2\times \mathbb{S}^1$
is constructed by inserting a blender in the Kan's example \eqref{Ex-diffeo}
and embedding into boundaryless manifold $\mathbb{T}^2\times \mathbb{S}^1$ (see \cite[Theorem 1.4]{CGS18}).

However, it is still unknown whether the Kan's original
example on $\mathbb{T}^2\times[0,1]$ is topologically transitive or not.
This paper is devoted to giving an affirmative answer
for Kan-type partially hyperbolic diffeomorphisms (see Definition \ref{def:Kan-type})
with a mild resonant condition originated from Kan's constructions.
By this result, we complete the final piece of topological indecomposability of Kan's examples.

\begin{maintheorem}\label{main-thm}
	Let $F:\mathbb{T}^2\times[0,1]\circlearrowleft$ be a
	Kan-type partially hyperbolic diffeomorphism
	with $C^2$ regularity, defined by
	\begin{displaymath}
		F(x,t)=(Ax,\phi_x(t)),
	\end{displaymath}
	where $A:\mathbb{T}^2\circlearrowleft$ is an Anosov toral automorphism
	fixing two points $p,q\in \mathbb{T}^2$ and
	\begin{displaymath}
		\dfrac{\ln\phi_p'(0)}{\ln\phi_q'(0)}\notin\mathbb{Q}.
	\end{displaymath}
	Then, $F$ is topologically transitive.
\end{maintheorem}

This result also holds for the case
$A$ is an Anosov diffeomorphism on $\mathbb{T}^n$ based on the  same argument.
We show just  the bare bone proof to make it more clear.
Comparing to the argument in \cite[Proposition 11.2]{BDV05},
we keep the original assumptions on the Kan's examples without additional ones.
Moreover, there is a property behind the Kan's examples,
called mostly contracting (see \cite{Kan94,ABV00,BV00}), which requires
\begin{displaymath}
	\int_{\mathbb{T}^2}\phi_x'(0)dx<0
	\quad\text{and}\quad
	\int_{\mathbb{T}^2}\phi_x'(1)dx<0.
\end{displaymath}
As stated in Definition \ref{def:Kan-type}, we do not need this condition either.

Naturally, the Example \eqref{Ex-diffeo} belongs to
the category of Kan-type partially hyperbolic diffeomorphisms,
so there is the following result as a corollary.

\begin{main-cor}\label{main-cor}
	The Kan's example on $\mathbb{T}^2\times[0,1]$
	\begin{displaymath}
		f(x,y,t)=\left(3x+y, 2x+y, t+\frac{t(1-t)}{32}\cos(2\pi x)\right)
	\end{displaymath}
	is topologically transitive.
\end{main-cor}
\begin{proof}[Proof of Corollary \ref{main-cor}]
	It suffices to prove the irrational condition of Theorem \ref{main-thm}.
	Note that there are two fixing points $p=({1}/{2},0)$ and $q=(0,0)$
	with $\phi_p'(0)={31}/{32}$ and $\phi_q'(0)={33}/{32}$,
	which have the irrational relation.
	If not, we have
	\begin{displaymath}
		\dfrac{\ln\phi_p'(0)}{\ln\phi_q'(0)}
		=\dfrac{\ln31-\ln32}{\ln33-\ln32}=-\dfrac{m}{n},
	\end{displaymath}
	where $m,n\in\mathbb{N}$ and $\gcd(m,n)=1$.
	Rewriting this equality, we just get
	\begin{displaymath}
		3^m\cdot 11^m\cdot 31^n=2^{5m+5n},
	\end{displaymath}
	which is an obvious contradiction by the fundamental theorem of arithmetic.
\end{proof}

\section{Preliminaries}\label{sec-pre}

In this section, we give some definitions and results
that will be used in this paper.

\subsection{Kan-type partially hyperbolic diffeomorphism}

At the beginning, we recall the definition of partial hyperbolicity.
A $C^1$ diffeomorphism $f: M\to M$ is called (absolutely) partially hyperbolic if
the tangent bundle admits a continuous $Df$-invariant splitting $TM=E^s\oplus E^c \oplus E^u$ such that
$E^s$ is uniformly contracting, $E^u$ is uniformly expanding and the center $E^c$ is intermediate.

Note that there are some characterizations in the Kan's original example \eqref{Ex-diffeo},
so we introduce the following general class of partially hyperbolic diffeomorphisms,
with attempting to capture the main features of Kan's construction on $\mathbb{T}^2\times[0,1]$.

Here recall the pole maps on the unit interval.	
Let $\phi:[0,1]\to[0,1]$ be a diffeomorphism fixing two endpoints:
\begin{itemize}
	\item $\phi$ is called a north–south pole map (abbr. NS-map)
	if $0<\phi'(0)<1<\phi'(1)$ and $\phi(t)<t$ for every $t\in(0,1)$.
	
	\item $\phi$ is called a south–north pole map (abbr. SN-map)
	if $0<\phi'(1)<1<\phi'(0)$ and	$\phi(t)>t$ for every $t\in(0,1)$.
\end{itemize}

\begin{defn}\label{def:Kan-type}
	Let $F:\mathbb{T}^2\times[0,1]\circlearrowleft$ be a $C^2$ skew-product defined by
	\begin{displaymath}
		F(x,t)=(Ax,\phi_x(t)),
	\end{displaymath}
	where $A$ is an Anosov toral automorphism on $\mathbb{T}^2$ fixing two points $p$ and $q$.
	$F$ is called a \emph{Kan-type partially hyperbolic diffeomorphism}
	if the followings are satisfied:
	\begin{enumerate}
		\item[($K_1$)] For any $x\in \mathbb{T}^2$, $\phi_x(0)=0$ and $\phi_x(1)=1$,
		i.e., $F$ preserves the boundary components.
		\item[($K_2$)] $\phi_p$ is an NS-map and $\phi_q$ is an SN-map.
		\item[($K_3$)] For any $(x,t)\in \mathbb{T}^2\times[0,1]$, $\|A^{-1}\|^{-1}< \phi_x'(t) <\|A\|$,
		i.e., $F$ is (absolutely) partially hyperbolic.
	\end{enumerate}
\end{defn}

\subsection{Regularity of the holonomy map}

For the sake of delicate analysis,
some regularity is required for the induced holonomy dynamics to a certain extent.
This leads to the following concepts.

A partially hyperbolic diffeomorphism $f:M\to M$ is called dynamically coherent if
there are two $f$-invariant foliations $\mathcal{F}^{cs}$ and $\mathcal{F}^{cu}$, with $C^1$ leaves,
tangent to $E^s\oplus E^c$ and $E^c\oplus E^u$, respectively.
The foliations $\mathcal{F}^{cs}$ and $\mathcal{F}^{cu}$ are called center-stable and center-unstable foliation, respectively.
A partially hyperbolic diffeomorphism $f:M\to M$ satisfies center bunching condition if
\begin{displaymath}
	\|Df|_{E^s}\| < \dfrac{m(Df|_{E^c})}{\|Df|_{E^c}\|}
	\quad\text{and}\quad
	\|Df|_{E^u}\| > \dfrac{\|Df|_{E^c}\|}{m(Df|_{E^c})}.
\end{displaymath}

Then, as shown in \cite{PSW97}, 
there is an important regularity result.

\begin{thm}[\cite{PSW97}, Theorem B]\label{Thm:C1-holonomy}
	Let $f$ be a $C^2$ partially hyperbolic diffeomorphism
	satisfying the dynamically coherent and center bunching conditions.
	Then, the holonomy map defined by the strong stable/unstable foliation,
	restricted to a center-stable/center-unstable leaf,
	is a $C^1$ local diffeomorphism.
\end{thm}

\begin{rmk}
	For a $C^2$ Kan-type partially hyperbolic diffeomorphism,
	the holonomy map defined by the strong stable/unstable foliation,
	restricted to a center-stable/center-unstable leaf,
	is a $C^1$ diffeomorphism locally.
	In fact, the dynamically coherent condition is satisfied
	by Item $(K_3)$ in Definition \ref{def:Kan-type} with the skew-product form,
	and the	center bunching condition is also satisfied
	by Item $(K_3)$ with the setting of one-dimensional center.
\end{rmk}

\subsection{Analysis in the one-dimensional dynamics}

Before we start to prove the topological transitivity of Kan-type systems,
we need to get some preparation on the one-dimensional center dynamics.

First, we introduce the following classical linearization theorem by Sternberg,
for the convenience of providing $C^1$-charts in the one-dimensional dynamics.

\begin{thm}[\cite{Na11}, Theorem 3.6.2]\label{SLT}
	Let $f$ be a $C^2$ diffeomorphism
	from a neighborhood containing $0$ in $\mathbb{R}^+$ onto its image.
	If $f'(0)=\alpha\neq1$, then there is a $C^1$ local diffeomorphism
	$h$ onto its image with $h'(0)=1$
	and $ h\circ f=\alpha\cdot h$ near $0$.
\end{thm}

Now we present the intersection result in a quantitative way
with the map $h$ working as the holonomy map later.
Note that the inverse of NS-map is SN-map and vice versa.

\begin{prop}\label{prop-Duminy-id}
	Let $h:[0,1]\to[0,1]$ be an orientation-preserving $C^1$ diffeomorphism.
	Assume that $f$ and $g$ are $C^2$ NS-maps
	satisfying $\ln \alpha$ and $\ln \beta$ are rationally independent,
	where $\alpha=f'(0)$ and $\beta=g'(0)$.
	
	Then, for any $x\in(0,1)$, the set
	\begin{displaymath}
		\left\lbrace f^{-k}\circ h^{-1}\circ g^l(x):k,l\in\mathbb{N}^*\right\rbrace
	\end{displaymath}
	is dense in $[0,1]$.
	
	In particular, for any intervals $I,J \subseteq [0,1]$,
	there exist infinitely many pairs of integers $k_n,l_n>0$ such that
	\begin{displaymath}
		\big( h\circ f^{k_n}(I) \big) \bigcap g^{l_n}(J) \neq \varnothing
	\end{displaymath}
	and 		
	\begin{displaymath}
		\dfrac{k_n}{l_n} \to \dfrac{\ln\beta}{\ln\alpha}
		\quad \text{as} \quad n\to \infty.
	\end{displaymath}
	
	Moreover, denote by $I=[a,b]$ and $J=[c,d]$, there exists $\rho>0$ such that
	\begin{displaymath}
		\dfrac{| f^{k_n}(I) \bigcap
			\big( h^{-1}\circ g^{l_n}(J) \big) |}{f^{k_n}(b)-f^{k_n+1}(b)}
		\geqslant\rho
		\quad\text{and}\quad	
		\dfrac{| \big( h\circ f^{k_n}(I) \big) \bigcap g^{l_n}(J)|}
		{g^{l_n}(d)-g^{l_n+1}(d)}
		\geqslant\rho.
	\end{displaymath}
\end{prop}

\begin{proof}[Proof of Proposition \ref{prop-Duminy-id}]
	By Sternberg linearization theorem,
	there exist two $C^1$ local diffeomorphisms $h_1,h_2$ such that
	\begin{displaymath}
		h_1'(0)=h_2'(0)=1,
	\end{displaymath}
	with the following conjugate equations holding:
	\begin{displaymath}
		\begin{split}
			h_1\circ f(t)&=\alpha \cdot h_1(t),\\
			h_2\circ g(t)&=\beta  \cdot h_2(t),
		\end{split}
	\end{displaymath} 	
	for any $t$ near the hyperbolic sink $0$.
	
	Pick a small $\delta >0$ such that the conjugate equations both hold on $[0,\delta]$.
	Then, there is a commutative diagram as follows.
	\begin{displaymath}
		\begin{gathered} \xymatrix{
				f:[0,\delta]\circlearrowleft \ar[r]^h \ar[d]_{h_1} &g:[0,\delta]\circlearrowleft  \ar[d]_{h_2}\\
				\alpha:[0,h_1(\delta)]\circlearrowleft \ar[r]^{\tilde{h}}  &\beta:[0,h_2(\delta)]\circlearrowleft \\}	
		\end{gathered}
	\end{displaymath}
	
	We are going to show there exist infinitely many intersection iterates at first.
	\begin{lem}\label{lem-InterInf}
		For any $x\in (0,\delta]$ and $I=[a,b]\subseteq [0,\delta]$,
		there exist infinitely many integers $k_n,l_n>0$ such that
		\begin{displaymath}
			g^{l_n}(x) \in h\circ f^{k_n}([a,b]),
		\end{displaymath}
		where these iterate numbers $k_n,l_n$ satisfy
		\begin{displaymath}
			\dfrac{k_n}{l_n} \to \dfrac{\ln\beta}{\ln\alpha}
			\quad\quad \text{as} \quad n\to \infty.
		\end{displaymath}
	\end{lem}
	\begin{proof}[Proof of Lemma \ref{lem-InterInf}]
		From the conjugate viewpoint,
		we are going to prove (see Figure \ref{IntMech})
		\begin{equation}\label{prop-aim-inter}
			\beta^l(\tilde{x}) \in \tilde{h}\circ \alpha^k([\tilde{a},\tilde{b}]),
		\end{equation}
		where $\tilde{x}=h_2(x)$, $\tilde{I}=h_1(I)=[\tilde{a},\tilde{b}]$
		and $\tilde{h}=h_2 \circ h \circ h^{-1}_1$.

		\begin{figure}[htbp]
			\centering
			\includegraphics[width=8.6cm]{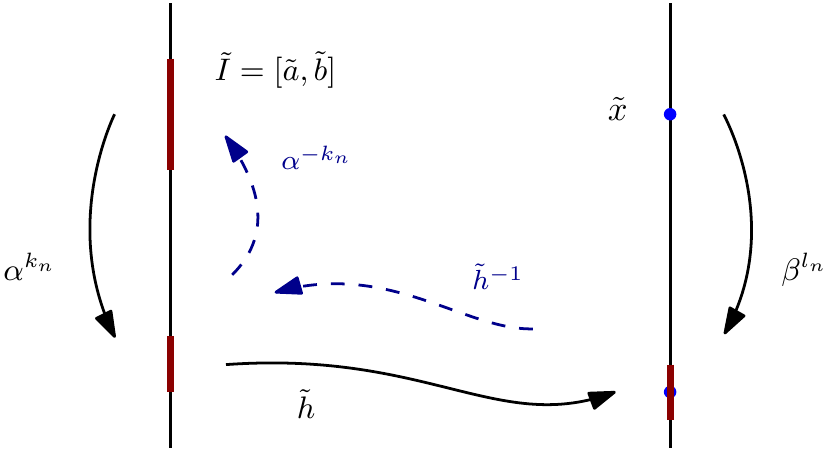}
			\caption{Intersections dynamics}
			\label{IntMech}
		\end{figure}

		Note that $\tilde{h}'(0)=h'(0)=\theta>0,$
		we can write
		\begin{displaymath}
			\tilde{h}^{-1}(t)=\theta^{-1} \cdot t + R(t),
		\end{displaymath}
		with the remainder term
		\begin{displaymath}
			|R(t)| \leqslant M(t)\cdot t
			\quad \text{for} \quad
			t\in [0,h_2(\delta)],
		\end{displaymath}
		where $M(t)\to 0$ as $t\to 0$.
		
		Then, for any $\tilde{x}\in(0,h_2(\delta)]$ and interval $\tilde{I}\subseteq[0,h_1(\delta)]$,
		there are two positive numbers
		$\eta=\eta(I)$ and $\varepsilon=\varepsilon(I)<\eta$ such that
		\begin{equation}\label{ToGene}
			[\eta^* \tilde{x}-\varepsilon,\eta^* \tilde{x}+\varepsilon]
			\subseteq \theta\cdot\tilde{I},
		\end{equation}
		for any positive number $\eta^*\in [\eta-\varepsilon,\eta+\varepsilon]$.
		
		Since $\ln \alpha$ and $\ln \beta$ are rationally independent,
		for the given number $\ln\eta$,
		we have infinitely many pairs integers $k_n,l_n\to\infty$ such that
		\begin{align}\label{ToGene=1}
			-k_n\ln\alpha+l_n\ln\beta \to \ln\eta,
			\quad \text{i.e.,} \quad
			\alpha^{-k_n}\cdot\beta^{l_n} \to \eta.
		\end{align}
		Thus, we can take $N>0$ such that, for all $n>N$,
		the corresponding infinitely many pairs integers $k_n,l_n$
		satisfy the inequality:
		\begin{align}\label{ToGene1}
			\alpha^{-k_n}\cdot\beta^{l_n}\triangleq\eta^*_n
			\in [\eta-\varepsilon,\eta+\varepsilon].
		\end{align}
		Meanwhile, note that $\beta<1$, for the integers $k_n,l_n\to\infty$, we have
		\begin{displaymath}
			|\alpha^{-k_n}\cdot R(\beta^{l_n}(\tilde{x}))|
			\leqslant \alpha^{-k_n}\cdot M(\beta^{l_n}\tilde{x})
			\cdot \beta^{l_n}\tilde{x} \to 0.
		\end{displaymath}
		We can enlarge this $N$ such that,
		for all $n>N$, these corresponding infinitely many pairs integers $k_n,l_n$
		also satisfy the inequality:
		\begin{align}\label{ToGene2}
			|\alpha^{-k_n}\cdot R(\beta^{l_n}(\tilde{x}))|
			< \theta^{-1}\cdot\varepsilon.
		\end{align}
		
		Hence, combining Inequalities
		\eqref{ToGene}, \eqref{ToGene1} and \eqref{ToGene2},
		we get
		\begin{displaymath}
			\begin{split}
				\alpha^{-k_n}\circ \tilde{h}^{-1}\circ \beta^{l_n} (\tilde{x})
				&= \alpha^{-k_n}(\theta^{-1}\cdot \beta^{l_n}(\tilde{x}) +
				R( \beta^{l_n}(\tilde{x}) ) ),\\
				&= \theta^{-1}\cdot\alpha^{-k_n}\cdot \beta^{l_n}(\tilde{x}) + \alpha^{-k_n}\cdot
				R( \beta^{l_n}(\tilde{x})),\\
				&\in\theta^{-1}\cdot[\eta^*_n \tilde{x}-\varepsilon,\eta^*_n \tilde{x}+\varepsilon]
				\subseteq \tilde{I}.
			\end{split}
		\end{displaymath}
		
		This implies that there exist infinitely many pairs integers $k_n,l_n>0$ such that
		the desired Inclusion \eqref{prop-aim-inter} holds,
		with these $k_n,l_n$ satisfying that
		\begin{displaymath}
			\dfrac{k_n}{l_n} \to \dfrac{\ln\beta}{\ln\alpha}
			\quad\quad \text{as} \quad n\to \infty.
		\end{displaymath}
	\end{proof}

	According to Lemma \ref{lem-InterInf},
	for any given intervals $I,J \subseteq [0,1]$,
	we naturally have
	\begin{displaymath}
		\big(h\circ f^{k_n}(I)\big) \bigcap g^{l_n}(J) \neq \varnothing.
	\end{displaymath}
	
	Moreover, we can make slight modifications for the choices of $k_n$ and $l_n$
	to get a uniform lower bound for the proportions
	between the length of these intersections
	and the length of the corresponding iterates of the fundamental domain.
	
	Precisely, for the given intervals $I,J \subseteq [0,1]$,
	we take their middle thirds and denote them by $I_0,J_0$, respectively.
	By the argument of the proof of Lemma \ref{lem-InterInf},
	for the corresponding number $\eta_0=\eta_0(I_0)>0$,
	there are also infinitely many pairs integers $k_n^0,l_n^0 >0$ with
	\begin{displaymath}
		\big(h\circ f^{k_n^0}(I_0)\big) \bigcap g^{l_n^0}(J_0) \neq \varnothing,
	\end{displaymath}	
	and
	\begin{displaymath}
		\alpha^{-k_n^0}\cdot\beta^{l_n^0} \to \eta_0
		\quad \text{as} \quad n\to \infty.
	\end{displaymath}
	By taking these iterates $k_n^0,l_n^0$
	on the previous intervals $I=[a,b],J=[c,d]$,
	we have the following refined result based on distortion control estimations.

	\begin{lem}\label{lem-InterSpe}
		There is a constant $\rho>0$ such that
		\begin{displaymath}
			\dfrac{|f^{k^0_n}(I) \bigcap
				\big(h^{-1}\circ g^{l^0_n}(J)\big) |}{f^{k^0_n}(b)-f^{k^0_n+1}(b)}
			\geqslant\rho
			\quad\text{and}\quad	
			\dfrac{| \big( h\circ f^{k^0_n}(I) \big)
				\bigcap g^{l^0_n}(J)|}{g^{l^0_n}(d)-g^{l^0_n+1}(d)}
			\geqslant\rho.
		\end{displaymath}
	\end{lem}
	\begin{proof}[Proof of Lemma \ref{lem-InterSpe}]
		It suffices to prove the existence of a constant $\rho$ with
		\begin{displaymath}
			\dfrac{| \big( h\circ f^{k^0_n}(I) \big)
				\bigcap g^{l^0_n}(J)|}{g^{l^0_n}(d)-g^{l^0_n+1}(d)}
			\geqslant\rho,
		\end{displaymath}
		since $h$ is an orientation-preserving diffeomorphism.
		For the intersection iterates of the middle intervals $I_0=[a_0,b_0]$ and $J_0=[c_0,d_0]$,
		there are only two possibilities
		for the intersections of the previous intervals $I$ and $J$ now.
		
		\vspace{4pt}\noindent	
		\textbf{Case 1. $h\circ f^{k_n^0}(I)$ is not totally contained in $g^{l_n^0}(J)$.}
		Without loss of generality, we assume $g(d)<c$. Then, we have
		\begin{displaymath}
			\left( g^{l_n^0}(c),g^{l_n^0}(c_0) \right) \subseteq h\circ f^{k_n^0}(I)
			\quad\text{or}\quad
			\left( g^{l_n^0}(d_0),g^{l_n^0}(d) \right) \subseteq h\circ f^{k_n^0}(I),
		\end{displaymath}
		since we have $k_n^0,l_n^0$
		with the intersection for middle thirds of the intervals $I,J$.
		
		Thus, by the distortion control argument,
		we can get the following estimations:
		\begin{displaymath}
			\begin{split}
				\dfrac{|\big(h\circ f^{k_n^0}(I) \big)
					\bigcap g^{l_n^0}(J)|}{g^{l_n^0}(d)-g^{l_n^0+1}(d)}
				&\geqslant \dfrac{g^{l_n^0}(c_0)-g^{l_n^0}(c)}
				{g^{l_n^0}(d)-g^{l_n^0+1}(d)} \\
				&\geqslant \dfrac{c_0-c}{d-g(d)}
				e^{\ln{(g^{l_n^0})'(\xi_1)}-\ln{(g^{l_n^0})'(\xi_2)}}\\
				&\geqslant
				\dfrac{c_0-c}{d-g(d)} e^{-G_{12} \cdot
					\sum\limits_{j=0}^{l_n^0-1}{|g^j(\xi_1)-g^j(\xi_2)|}}\\
				&\geqslant
				\dfrac{c_0-c}{d-g(d)} e^{-G_{12}\cdot d}\triangleq \rho_c,
			\end{split}
		\end{displaymath}
		or
		\begin{displaymath}
			\begin{split}
				\dfrac{|\big(h\circ f^{k_n^0}(I) \big)
					\bigcap g^{l_n^0}(J)|}{g^{l_n^0}(d)-g^{l_n^0+1}(d)}
				&\geqslant \dfrac{g^{l_n^0}(d)-g^{l_n^0}(d_0)}
				{g^{l_n^0}(d)-g^{l_n^0+1}(d)} \\
				&\geqslant \dfrac{d-d_0}{d-g(d)}
				e^{\ln{(g^{l_n^0})'(\xi_3)}-\ln{(g^{l_n^0})'(\xi_4)}}\\
				&\geqslant
				\dfrac{d-d_0}{d-g(d)} e^{-G_{12} \cdot
					\sum\limits_{j=0}^{l_n^0-1}{|g^j(\xi_3)-g^j(\xi_4)|}}\\
				&\geqslant
				\dfrac{d-d_0}{d-g(d)} e^{-G_{12}\cdot d}\triangleq \rho_d,
			\end{split}
		\end{displaymath}
		where
		\begin{displaymath}
			G_{12}=\dfrac{\max\limits_{t\in[0,\tilde{\delta}]}
				{|g''(t)|}}{\min\limits_{t\in[0,\tilde{\delta}]}{|g'(t)|}}
			\quad \text{and} \quad 
			\xi_i \in \big(g(d),d\big) \quad \text{for} \quad i=1,2,3,4.
		\end{displaymath}
		
		\vspace{4pt}\noindent	
		\textbf{Case 2. $h\circ f^{k_n^0}(I)$ is exactly totally contained in $g^{l_n^0}(J)$.}
		So we  have
		\begin{displaymath}
			\dfrac{|\big(h\circ f^{k_n^0}(I)\big) \bigcap
				g^{l_n^0}(J)|}{g^{l_n^0}(d)-g^{l_n^0+1}(d)}	=
			\dfrac{|h\circ f^{k_n^0}(I)|}{g^{l_n^0}(d)-g^{l_n^0+1}(d)}.
		\end{displaymath}
		
		Similarly, we deal with the situation
		from the conjugate viewpoint at first.
		Denote by $\tilde{I}=h_1(I)=[\tilde{a},\tilde{b}]$
		and $\tilde{J}=h_2(J)=[\tilde{c},\tilde{d}]$.
		By $\tilde{h}'(0)>0$,
		there exists $\theta'>0$ with
		\begin{displaymath}
			\tilde{h}'(t)\geqslant \theta' \quad \text{for} \quad t\in [0,h_1(\delta)].
		\end{displaymath}
		Note that $\alpha<1$, we can always take $k_n^0$ large enough to guarantee
		\begin{displaymath}
			\alpha^{k_n^0}(\tilde{b})<h_1(\delta),
		\end{displaymath}
		which implies
		\begin{displaymath}
			\tilde{h}(\alpha^{k_n^0}(\tilde{b}))-\tilde{h}(\alpha^{k_n^0}(\tilde{a}))
			\geqslant \theta'\cdot \alpha^{k_n^0}\cdot(\tilde{b}-\tilde{a}).
		\end{displaymath}
		At the same time, by taking $k_n^0,l_n^0$ large enough,
		we can get a number $\eta_0'\in(0,\eta_0^{-1})$ satisfying
		\begin{displaymath}
			\alpha^{k_n^0}\cdot\beta^{-l_n^0}\geqslant \eta_0'.
		\end{displaymath}
		In fact, this also comes from Limit \eqref{ToGene=1}, that is,
		\begin{displaymath}
			\alpha^{k_n^0}\cdot\beta^{-l_n^0} \to
			\eta^{-1} \quad \text{as} \quad k_n^0,l_n^0\to\infty.
		\end{displaymath}
		Thus, we can get the following estimation:
		\begin{displaymath}
			\begin{split}
				\dfrac{|\tilde{h}\circ \alpha^{k_n^0}(\tilde{I})|}{\beta^{l_n^0}(\tilde{d})-\beta^{l_n^0+1}(\tilde{d})}
				&= \dfrac{\tilde{h}\circ \alpha^{k_n^0}(\tilde{b})
					-\tilde{h}\circ \alpha^{k_n^0}(\tilde{a}) }
				{\beta^{l_n^0}(\tilde{d})-\beta^{l_n^0+1}(\tilde{d})}\\
				&\geqslant \dfrac{\theta'\cdot \alpha^{k_n^0}\cdot(\tilde{b}-\tilde{a}) }
				{\beta^{l_n^0}\cdot(\tilde{d}-\beta\tilde{d})}\\	
				&\geqslant \theta'\eta_0'\dfrac{\tilde{b}-\tilde{a}}{\tilde{d}-\beta\tilde{d}}
				\triangleq \tilde{\rho_0}.
			\end{split}
		\end{displaymath}
		
		Because $h_2'(0)=1$ and $0$ is the sink,
		taking $\delta$ small enough, we have $\theta_2'>0$ with
		\begin{displaymath}
			\dfrac{\min\limits_{t\in[0,h_2(\delta)]}{(h_2^{-1})'(t)}}
			{\max\limits_{t\in[0,h_2(\delta)]}{(h_2^{-1})'(t)}} \geqslant \theta_2'.
		\end{displaymath}
		Moreover, we get
		\begin{displaymath}
			\begin{split}
				\dfrac{|\big(h\circ f^{k_n^0}(I)\big)
					\bigcap g^{l_n^0}(J)|}{g^{l_n^0}(d)-g^{l_n^0+1}(d)}
				&=	\dfrac{|h\circ f^{k_n^0}(I)|}{g^{l_n^0}(d)-g^{l_n^0+1}(d)}\\
				&=	\dfrac{|h_2^{-1}\circ\tilde{h}\circ \alpha^{k_n^0}(\tilde{I})|}
				{h_2^{-1}\circ(\beta^{l_n^0}(\tilde{d}))-h_2^{-1}
					\circ(\beta^{l_n^0+1}(\tilde{d}))}\\
				&\geqslant
				\theta_2'\dfrac{|\tilde{h}\circ \alpha^{k_n^0}(\tilde{I})|}
				{\beta^{l_n^0}(\tilde{d})-\beta^{l_n^0+1}(\tilde{d})}\\
				&\geqslant
				\theta_2'\cdot\tilde{\rho_0}\triangleq \rho_0.
			\end{split}
		\end{displaymath}
		
		Finally, combining two cases, we can take
		\begin{displaymath}
			\rho=\min\left\{\rho_0,\rho_c,\rho_d\right\}>0,
		\end{displaymath}
		then we obtain the desired uniform lower bound $\rho$ satisfying
		\begin{displaymath}	
			\dfrac{|\big(h\circ f^{k_n^0}(I)\big) \bigcap g^{l_n^0}(J)|}{g^{l_n^0}(d)-g^{l_n^0+1}(d)}\geqslant\rho.
		\end{displaymath}
	\end{proof}
	
	Combining these two lemmas,
	we just complete the proof of Proposition \ref{prop-Duminy-id}.	
\end{proof}

\section{Kan-type transitivity}

Now we present the proof of Theorem \ref{main-thm} in detail
to show Kan-type transitivity.

\begin{figure}[htbp]
	\centering
	\includegraphics[width=6.6cm]{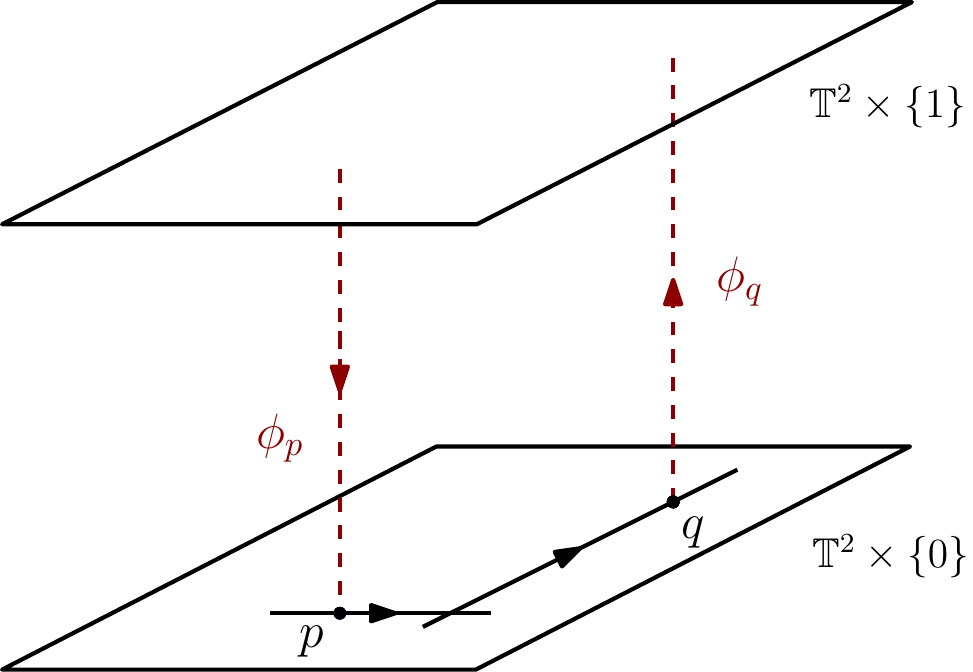}
	\caption{Kan-type dynamics}
	\label{KanDyn}
\end{figure}

\begin{proof}[Proof of Theorem \ref{main-thm}]
	
	For any given non-empty open sets $U,V\subseteq\mathbb{T}^2\times[0,1]$,
	we are going to show the existence of some positive integer $m$ with
	\begin{displaymath}
		F^m(U)\cap V\neq\varnothing.
	\end{displaymath}
	To be clear, we divide the proof into three steps.
	
	\vspace{4pt}\noindent
	\textbf{{Step 1. Holonomy maps with specific sizes.}}

	At first, we define the holonomy map on the torus.
	By the hyperbolicity of $A$, there are transverse $A$-invariant foliations
	$\mathcal{L}^u$ and $\mathcal{L}^s$ on $\mathbb{T}^2$.
	Denote by $h^s_p$ the holonomy map near $p$ along the stable manifolds,
	\begin{displaymath}
		h^s_p:\mathcal{L}^u_{loc}(p)\to\mathcal{L}^u_{loc}(\tilde{p}), ~~ h^s_p(x)=\mathcal{L}_{loc}^s(x)\cap\mathcal{L}^u_{loc}(\tilde{p}),
	\end{displaymath}
	for any $\tilde{p}\in\mathcal{L}_{loc}^s(p)$.
	
	Since the partially hyperbolic diffeomorphism $F$ has interval center fibers,
	for $F$-invariant strong unstable manifolds $\mathcal{W}^u$
	and strong stable manifolds $\mathcal{W}^s$ on $\mathbb{T}^2\times[0,1]$,
	there is a local map $H^s_p$ projecting to $h^s_p$.
	We give this construction carefully.
	
	For simplicity, we take the abusing symbols between $(p,0)$ and $p$
	without making confusions on understanding,
	by regarding $\Gamma_0=\mathbb{T}^2\times\{0\}$ as $\mathbb{T}^2$.
	We denote by $I_p$ the center interval leaf containing the point $p$
	and denote by
	\begin{displaymath}
		\pi:\mathbb{T}^2\times[0,1]\to\Gamma_0
	\end{displaymath}
	the canonical projection along the center direction.
	
	Note that
	\begin{displaymath}
		\overline{\bigcup_{k>0} F^{-k}(\mathcal{W}_{loc}^s(I_p))}
		= \mathbb{T}^2\times[0,1],
	\end{displaymath}
	for the given open set $U\subseteq\mathbb{T}^2\times[0,1]$,
	we take an integer $k_0^{s}>0$ such that
	\begin{displaymath}
		F^{k_0^{s}}(U)\cap\mathcal{W}^s_{loc}(I_p)\neq\varnothing.
	\end{displaymath}
	
	Moreover, taking a point
	\begin{displaymath}
		\tilde{p}\in \left(\pi\circ F^{k^s_0}(U)\right)\cap\mathcal{L}^s(p) \subseteq \Gamma_0,
	\end{displaymath}
	we redefine the holonomy map $h_p^{s}$ on $\Gamma_0$
	with the specific sizes $\varepsilon_p$ and $\tilde{\varepsilon_{p}}$,
	\begin{displaymath}
		\begin{split}
			&h^s_p:\mathcal{L}^u_{\varepsilon_p}(p)\to
			\mathcal{L}^u_{\tilde{\varepsilon_p}}(\tilde{p}),\\
			&h^s_p(x)=\mathcal{L}_{loc}^s(x)\cap
			\mathcal{L}^u_{\tilde{\varepsilon_p}}(\tilde{p}).
		\end{split}
	\end{displaymath}
	Here we have an interval $\tilde{J}_{\tilde{p}}\subseteq I_{\tilde{p}} \cap F^{k^s_0}(U)$.
	
	Then, choosing $\delta_u < \varepsilon_p $,
	we have
	\begin{displaymath}
		\mathcal{W}^u_{\delta_u}(I_p)\subseteq\mathcal{L}^u_{\varepsilon_p}(p) \times [0,1],
	\end{displaymath}
	and we define the holonomy map $H_p^s$ on $\mathbb{T}^2\times[0,1]$
	along the strong stable manifolds by
	\begin{displaymath}
		\begin{split}
			&H^s_p:\mathcal{L}^u_{\varepsilon_p}(p) \times [0,1]
			\to\mathcal{L}^u_{\tilde{\varepsilon_p}}(\tilde{p}) \times [0,1],\\
			&H^s_p(z)=\mathcal{W}_{loc}^s(z)\cap
			\mathcal{L}^u_{\tilde{\varepsilon_p}}(\tilde{p}) \times [0,1].
		\end{split}
	\end{displaymath}
	Here we have an interval $J_p\subseteq I_p$ with $H_p^s(J_p)=\tilde{J}_{\tilde{p}}$.
	
	Thus, denote by
	\begin{displaymath}
		U_0 = F^{k^s_0}(U)
	\end{displaymath}
	and take
	\begin{displaymath}
		\delta_p=\min\limits_{t\in I_p} \max
		\limits_{z\in\mathcal{W}^u_{\delta_u}(t)}\{d(\pi(z),p)\}.	
	\end{displaymath}
	Then, by decreasing $\delta_u$ (hence $\delta_p$), we get
	every center interval $J_{p'}$ in $\Gamma^{cu}_p\triangleq\mathcal{W}^u_{\delta_p}(J_p)$ satisfies
	\begin{displaymath}
		\tilde{J}_{\tilde{p'}}\triangleq H_p^s(J_{p'})\subseteq U_0
		\quad \text{and} \quad
		\Gamma^{cu}_{\tilde{p'}}\triangleq
		\mathcal{W}^u_{\delta_p}(\tilde{J}_{\tilde{p'}}) \subseteq U_0,
	\end{displaymath}
	where $p'\in \mathcal{L}^u_{\delta_p}(p)$ and $\tilde{p'}=\pi\circ H_p^s(p')$.
	
	In the same manner as above, we deal with the other part.
	For the given open set $V\subseteq\mathbb{T}^2\times[0,1]$, we choose $l_0^u>0$ with
	\begin{displaymath}
		F^{-l_0^{u}}(V)\cap\mathcal{W}^u_{loc}(I_q)\neq\varnothing.
	\end{displaymath}
	Denote by
	\begin{displaymath}
		V_0 = F^{-l_0^{u}}(V).
	\end{displaymath}	
	Then, we take
	\begin{displaymath}
		\tilde{q}\in \left(\pi\circ V_0\right) \cap \mathcal{L}^u(q)
	\end{displaymath}
	with corresponding sizes $\varepsilon_q, \tilde{\varepsilon_q}$
	such that the following holonomy map $H^u_q$ is well-defined
	\begin{displaymath}
		\begin{split}
			&H^u_q:\mathcal{L}^s_{\varepsilon_q}(q) \times [0,1]
			\to\mathcal{L}^s_{\tilde{\varepsilon_q}}(\tilde{q}) \times [0,1],\\
			&H^u_q(z)=\mathcal{W}_{loc}^u(z)\cap
			\mathcal{L}^s_{\tilde{\varepsilon_q}}(\tilde{q}) \times [0,1].
		\end{split}
	\end{displaymath}
	Moreover, for the interval $\tilde{J}_{\tilde{q}}\subseteq I_{\tilde{q}} \cap V_0$,
	we have $H_q^u(J_q)=\tilde{J}_{\tilde{q}}$ with some interval $J_q\subseteq I_q$.
	Thus, take $\delta_s<\varepsilon_q$ and denote by
	\begin{displaymath}
		\delta_q=\min\limits_{t\in I_q} \max\limits_{z\in\mathcal{W}^u_{\delta_s}(t)}\{d(\pi(z),q)\}.	
	\end{displaymath}
	Then, by decreasing $\delta_s$ (hence $\delta_q$), we get
	every center interval $J_{q'}$
	in $\Gamma^{cs}_q\triangleq\mathcal{W}^s_{\delta_q}(J_q)$ satisfies
	\begin{displaymath}
		\tilde{J}_{\tilde{q'}}\triangleq H_q^u(J_{q'})\subseteq V_0
		\quad \text{and} \quad
		\Gamma^{cs}_{\tilde{q'}}\triangleq
		\mathcal{W}^s_{\delta_q}(\tilde{J}_{\tilde{q'}}) \subseteq V_0,
	\end{displaymath}
	where $q'\in \mathcal{L}^s_{\delta_q}(q)$ and $\tilde{q'}=\pi\circ H_q^u(q')$.
	
	Here we mention that
	before picking the sizes $\delta_u,\delta_s$,
	we take the positive integers $k_0^s,l_0^u$ large enough such that
	\begin{displaymath}
		\begin{split}
			d_s(p,\tilde{p})&<d_c(b,\phi_p(b)),\\
			d_u(q,\tilde{q})&<d_c(d,\phi_q^{-1}(d)),
		\end{split}
	\end{displaymath}
	where we denote by $J_p=[a,b]$ and $J_q=[c,d]$.
	In fact, this comes from the partial hyperbolicity,
	which means there exist $\lambda,\gamma\in(0,1)$ such that
	\begin{displaymath}
		\|A^{-1}\|^{-1}=\lambda<\gamma< \phi_x'(t)<\gamma^{-1} <\lambda^{-1}=\|A\|,
	\end{displaymath}
	for any $(x,t)\in\mathbb{T}^2\times [0,1]$.
	At the same time, note that
	the center fiber is straight and $\phi_p,\phi_q^{-1}$ are NS-maps,
	we can also take the positive integers $k_0^s,l_0^u$ large enough such that
	there exists $Q>0$ satisfying
	\begin{displaymath}
		\begin{split}
			d_s(x,H_p^s(x)) &\leqslant Q\cdot d_s(\pi(x),\pi(H_p^s(x))),\\
			d_u(y,H_q^u(y)) &\leqslant Q\cdot d_u(\pi(y),\pi(H_q^u(y))),
		\end{split}
	\end{displaymath}
	where $x,y$ are near the boundary $\Gamma_0$ with
	\begin{displaymath}
		d_c(x,\pi(x))<b
		\quad \text{and} \quad
		d_c(y,\pi(y))<d .
	\end{displaymath}
	
	\vspace{4pt}\noindent	
	\textbf{{Step 2. Intersections in center dynamics and the projection.}}

	Now we apply Proposition \ref{prop-Duminy-id} to the following setting.
	Take
	\begin{displaymath}
		r\in\mathcal{L}^u(p)\pitchfork\mathcal{L}^s(q)
	\end{displaymath}
	and denote the holonomy maps from $p$ to $q$ by
	\begin{displaymath}
		\begin{split}
			H^u_p:\mathcal{L}^s_{loc}(p) \times [0,1]
			\to\mathcal{L}^s_{loc}(r) \times [0,1],\\
			H^s_q:\mathcal{L}^u_{loc}(r) \times [0,1]
			\to\mathcal{L}^u_{loc}(q) \times [0,1].
		\end{split}
	\end{displaymath}
	Then, for $C^2$ NS-maps $\phi_p=f$, $\phi_q^{-1}=g$
	and $C^1$ local diffeomorphism
	\begin{displaymath}
		H\triangleq H_q^s\circ H_p^u=h,
	\end{displaymath}
	we have infinitely many pairs of integers $k_n,l_n>0$ such that 	
	\begin{align}\label{Intersection-bd-1}
		(H\circ \phi_p^{k_n}(J_p)) \cap \phi_q^{-l_n}(J_q) \neq \varnothing.
	\end{align}
	We also have $\rho >0 $ satisfying
	\begin{align}\label{Intersection-bd-1-left}
		\dfrac{| \phi_p^{k_n}(J_p) \cap (H^{-1}\circ\phi_q^{-l_n}(J_q)) |}
		{\phi_p^{k_n}(b)-\phi_p^{k_n+1}(b)}\geqslant \rho
	\end{align}
	and
	\begin{align}\label{Intersection-bd-1-right}
		\dfrac{| (H\circ\phi_p^{k_n}(J_p)) \cap \phi_q^{-l_n}(J_q) |}
		{\phi_q^{-l_n}(d)-\phi_q^{-l_n-1}(d)}\geqslant \rho.
	\end{align}
	
	In other words, for the $k_n,l_n$ large enough with
	\begin{displaymath}
		\begin{split}
			\lambda^{-k_n}\cdot\delta_p &>2d_u(p,r), \\
			\lambda^{-l_n}\cdot\delta_q &>2d_s(q,r),
		\end{split}
	\end{displaymath}
	we can apply Proposition \ref{prop-Duminy-id} to get the following intersection
	by  Item \eqref{Intersection-bd-1}:
	\begin{displaymath}
		F^{k_n}(\Gamma_p^{cu}) \cap F^{-l_n}(\Gamma_q^{cs}) \neq \varnothing.
	\end{displaymath}
	Thus, there exists a center interval $J_r$ in $I_r$ such that
	\begin{align}\label{Intersection-Ir}
		J_r= F^{k_n}(J_{p'}) \cap F^{-l_n}(J_{q'}),
	\end{align}
	for center intervals $J_{p'} $ in $\mathcal{W}^u_{\delta_p}(J_p)=\Gamma^{cu}_p$
	and $J_{q'}$ in $\mathcal{W}^s_{\delta_q}(J_q)=\Gamma^{cs}_q$.

	Moreover, since the choices of $\Gamma^{cu}_p$ and $\Gamma^{cs}_q$ satisfy
	\begin{displaymath}
		\mathcal{W}^u_{\delta_p}(H_p^s(J_{p'})) \subseteq U_0
		\quad \text{and} \quad
		\mathcal{W}^s_{\delta_q}(H_q^u(J_{q'})) \subseteq V_0,
	\end{displaymath}
	so under the iterates of $F$, we actually get
	from Equality \eqref{Intersection-Ir}:
	\begin{displaymath}
		\begin{split}
			F^{k_n}(\Gamma_p^{cu}) \cap F^{-l_n}(V_0) \neq \varnothing,\\
			F^{k_n}(U_0) \cap F^{-l_n}(\Gamma_q^{cs}) \neq \varnothing.
		\end{split}
	\end{displaymath}
	Here we also obtain $F^{k_n}(U_0)$ and $F^{-l_n}(V_0)$ intersect
	under the natural projection of $\pi$,
	that is, there exists $\tilde{r}\in\Gamma_0$
	near $r$ (see Figure \ref{HolProj}) satisfying
	\begin{displaymath}
		\tilde{r}
		=   \mathcal{L}^u_{loc}(r_s) \pitchfork \mathcal{L}^s_{loc}(r_u)
		~\in~ \big(\pi\circ F^{k_n}(U_0)\big) \cap \big(\pi\circ F^{-l_n}(V_0)\big),
	\end{displaymath}
	where
	\begin{displaymath}
		\begin{split}
			r_s &\triangleq A^{k_n}(\tilde{p'}) = A^{k_n}(h_p^s(p')),\\
			r_u &\triangleq A^{-l_n}(\tilde{q'})= A^{-l_n}(h_q^u(q')).
		\end{split}
	\end{displaymath}

	\begin{figure}[htbp]
		\centering
		\includegraphics[width=6.6cm]{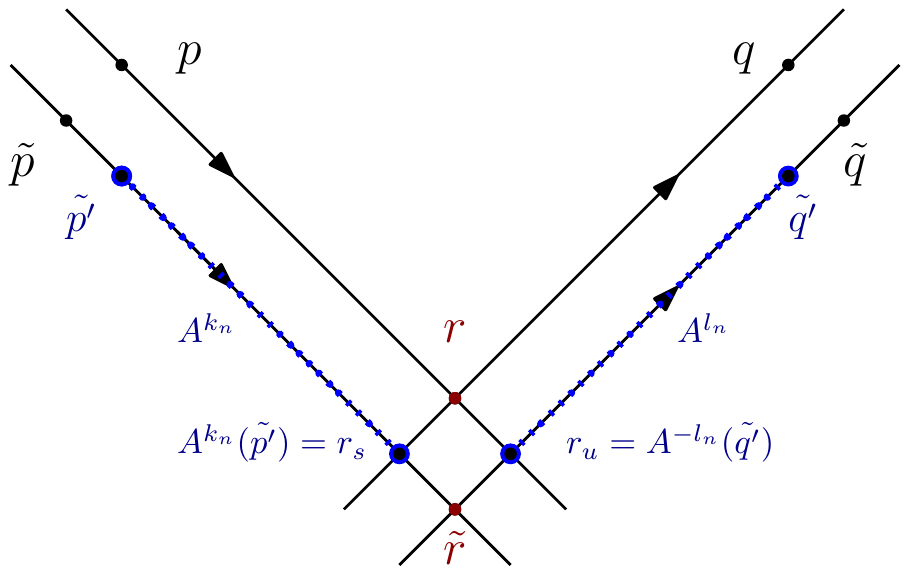}
		\caption{Dynamics under the projection}
		\label{HolProj}
	\end{figure}
	
	\noindent	
	\textbf{{Step 3. Distortion analysis by the partial hyperbolicity.}}

	Finally, we show the contradiction
	if there exists large distortion in the center direction.
	Denote by
	\begin{displaymath}
		J_{p'}=[a',b'] \quad\text{and}\quad J_{q'}=[c',d'].
	\end{displaymath}
	From Equality \eqref{Intersection-Ir}, without loss of generality,
	we suppose in center leaf $I_r$:
	\begin{displaymath}	
		c_r<a_r < d_r<b_r,
	\end{displaymath}
	where
	\begin{displaymath}
		[a_r,b_r]=F^{k_n}([a',b'])
		\quad\text{and}\quad [c_r,d_r]=F^{-l_n}([c',d']).
	\end{displaymath}
	In what follows, we are going to prove
	\begin{displaymath}
		F^{k_n}(U_0) \cap F^{-l_n}(V_0) \neq\varnothing.
	\end{displaymath}
	
	Otherwise, we can assume the order in $I_{\tilde{r}}$:
	\begin{align}\label{Contradic-0}
		{\tilde{a_r}}'=\mathcal{W}^u_{loc}(\tilde{a_r})\cap I_{\tilde{r}}
		> {\tilde{d_r}}'=\mathcal{W}^s_{loc}(\tilde{d_r})\cap I_{\tilde{r}},
	\end{align}
	where
	\begin{displaymath}
		\tilde{a_r} \triangleq H_p^s(a_r)\in I_{r_s}
		\quad\text{and}\quad
		\tilde{d_r} \triangleq H_q^u(d_r)\in I_{r_u}.
	\end{displaymath}
	Then, we have $e_r>d_r$ in $I_r$ (see Figure \ref{HolDist}) such that
	\begin{displaymath}
		\begin{split}
			H_q^u(e_r) =\tilde{e_r}>\tilde{d_r}
			\quad \text{and} \quad
			{\tilde{a_r}}'=\mathcal{W}^s_{loc}(\tilde{e_r})\cap I_{\tilde{r}}.
		\end{split}
	\end{displaymath}
	Note that $e_r>d_r>a_r$, that is, $d_c(a_r,e_r)>d_c(a_r,d_r),$
	so we get
	\begin{align}\label{Contradic+0}
		&d_s(a_r,\tilde{a_r}) + d_u(\tilde{a_r},{\tilde{a_r}}')
		+ d_s({\tilde{a_r}}',\tilde{e_r}) + d_u(\tilde{e_r},e_r) \nonumber\\
		&> d_c(a_r,e_r)  >d_c(a_r,d_r).
	\end{align}	
	\begin{figure}[htbp]
		\centering
		\includegraphics[width=8.6cm]{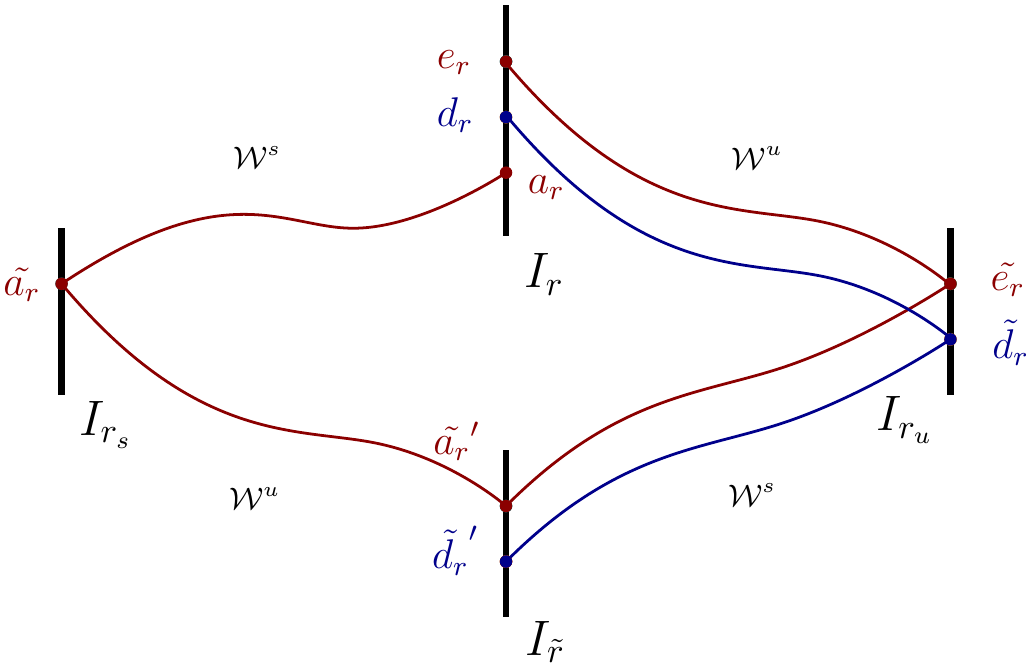}
		\caption{Strong holonomy and center distortion}
		\label{HolDist}
	\end{figure}

	On the one hand, we have
	\begin{displaymath}
		d_c(a_r,d_r) = |(H_p^u\circ \phi_p^{k_n}(J_p))
		\cap (H_q^s\circ\phi_q^{-l_n}(J_q))|.
	\end{displaymath}
	By Inequalities \eqref{Intersection-bd-1-left} and \eqref{Intersection-bd-1-right},
	there exist constants $K_1$ and $K_2$
	which are only dependent on $H_p^u$ and $H_q^s$, respectively, such that
	\begin{align}\label{Intersection-bd+1-left}
		d_c(a_r,d_r)
		&\geqslant	K_1\cdot | \phi_p^{k_n}(J_p)
		\cap (H^{-1}\circ\phi_q^{-l_n}(J_q)) | \nonumber\\
		&\geqslant  K_1\cdot \rho
		\cdot (\phi_p^{k_n}(b)-\phi_p^{k_n+1}(b)) \nonumber\\
		&>			K_1\rho\gamma^{k_n}d_c(b,\phi_p(b)) \triangleq D_1^n
	\end{align}
	and		
	\begin{align}\label{Intersection-bd+1-right}
		d_c(a_r,d_r)
		&\geqslant	K_2\cdot | (H\circ\phi_p^{k_n}(J_p))
		\cap \phi_q^{-l_n}(J_q) | \nonumber\\
		&\geqslant	K_2\cdot \rho \cdot	
		(\phi_q^{-l_n}(d)-\phi_q^{-l_n-1}(d))	\nonumber\\
		&> 			K_2\rho\gamma^{l_n}d_c(d,\phi_q^{-1}(d)) \triangleq D_2^n.
	\end{align}
	Here, recall that the constant $\gamma\in(0,1)$ satisfies
	$\phi_x'(t)\in (\gamma,\gamma^{-1})$.
	
	On the other hand, we can get
	\begin{displaymath}
		\begin{split}
			&d_s(a_r,\tilde{a_r}) + d_u(\tilde{a_r},{\tilde{a_r}}')
			+ d_s({\tilde{a_r}}',\tilde{e_r}) +  d_u(\tilde{e_r},e_r) \\
			&\leqslant	Q\cdot  \{ d_s(r,r_s) + d_u(r_s,\tilde{r})
			+ d_s(\tilde{r},r_u)              +  d_u(r_u, r) \}       \\
			&\leqslant			2Q\cdot \{ d_s(r,r_s) + d_u(r_u, r) \}\\
			&\leqslant	2Q\cdot \{ d_s(A^{k_n}(p'),A^{k_n}(\tilde{p'}))
			+ d_u(A^{-l_n}(q'),A^{-l_n}(\tilde{q'})) \} \\
			&\leqslant  2Q\cdot \{ \lambda^{k_n} d_s(p',\tilde{p'}) 	
			+ \lambda^{l_n} d_u(q',\tilde{q'}) 	     \} \\
			&\leqslant			2Q\cdot \{ \lambda^{k_n} d_s(p,\tilde{p})		
			+ \lambda^{l_n}d_u(q,\tilde{q})		     \} ,
		\end{split}
	\end{displaymath}
	that is, we have
	\begin{align}\label{Contradic-2}
		& d_s(a_r,\tilde{a_r}) + d_u(\tilde{a_r},{\tilde{a_r}}')
		+ d_s({\tilde{a_r}}',\tilde{e_r})
		+ d_u(\tilde{e_r},e_r)\nonumber\\
		& \leqslant  2Q\cdot\lambda^{k_n} d_s(p,\tilde{p})
		+ 2Q\cdot\lambda^{l_n}d_u(q,\tilde{q})
		\triangleq R_1^n+R_2^n.
	\end{align}
	Here, note that the $A$-invariant foliations
	$\mathcal{L}^u$ and $\mathcal{L}^s$ on $\Gamma_0$ are parallel lines,
	and recall that the uniform constant $Q$
	comes from the choices of $k_0^s$ and $l_0^u$ above.
	
	By the partial hyperbolicity of $F$, we have $\lambda<\gamma$,
	so we will get a contradiction when taking $k_n,l_n$ large enough.
	In fact, by the choices of $k_0^s$ and $l_0^u$, we have
	\begin{displaymath}
		d_s(p,\tilde{p})<d_c(b,\phi_p(b))
		\quad \text{and} \quad
		d_u(q,\tilde{q})<d_c(d,\phi^{-1}_q(d)).
	\end{displaymath}
	Note that $\lambda^n/\gamma^n\to 0$ and
	these constants $\rho,K_1,K_2,Q$ are all independent of $k_n,l_n$,
	so we can take $k_n,l_n$ large enough such that
	\begin{displaymath}
		D_i^n\geqslant 2\cdot R_i^n
		\quad \text{for} \quad i=1,2.
	\end{displaymath}
	Thus, from Inequalities \eqref{Intersection-bd+1-left},
	\eqref{Intersection-bd+1-right} and \eqref{Contradic-2},
	we will get
	\begin{displaymath}
		\begin{split}
			&d_s(a_r,\tilde{a_r}) + d_u(\tilde{a_r},{\tilde{a_r}}')
			+ d_s({\tilde{a_r}}',\tilde{e_r}) + d_u(\tilde{e_r},e_r)\\
			&\leqslant R_1^n + R_2^n
			\leqslant 2\cdot\max\limits_{i=1,2}\{R_i^n\}
			\leqslant \max\limits_{i=1,2}\{D_i^n\}
			< d_c(a_r,d_r),
		\end{split}
	\end{displaymath}	
	which exactly contradicts Inequality \eqref{Contradic+0}.
	So, Assumption \eqref{Contradic-0} does not hold and we obtain
	\begin{displaymath}
		F^{k_n}(U_0) \cap F^{-l_n}(V_0) \neq\varnothing.
	\end{displaymath}
	
	Hence, for the given $U,V$, by taking some $k_n,l_n$ large enough,
	we obtain the desired positive integer
	\begin{displaymath}
		m=k_n+k_0^s+l_n+l_0^u,
	\end{displaymath}
	such that
	\begin{displaymath}
		F^{m}(U) \cap V \neq\varnothing.
	\end{displaymath}
	
	This ends the proof of Theorem \ref{main-thm}.
\end{proof}

\section*{Acknowledgments}
This work is supported by
China Postdoctoral Science Foundation (2023M731527)
and National Key R\&D Program of China (2021YFA1001900).
And the author is thankful for the support from Ra\'{u}l Ures and NNSFC 12071202 and 12161141002.
The author also appreciates the referees for their valuable advice to improve the presentation of this paper.

\bibliographystyle{plain}
\bibliography{Bib-Kan2023}

\begin{thebibliography}{99}
	
	\bibitem[1]{ABV00} (MR1757000) [10.1007/s002220000057]
	\newblock J. F. Alves, Ch. Bonatti and M. Viana,
	\newblock \emph{S{RB} measures for partially hyperbolic systems whose central direction is mostly expanding},
	\newblock \emph{Invent. Math.}, \textbf{140} (2000), 351-398.
	
	\bibitem[2]{BDV05} (MR2105774)
	\newblock Ch. Bonatti, L. J. D\'{\i}az and M. Viana,
	\newblock \emph{Dynamics Beyond Uniform Hyperbolicity},
	\newblock Springer-Verlag, Berlin, 2005.
	
	\bibitem[3]{BP18} (MR3799757) [10.1007/s11856-018-1648-6]
	\newblock Ch. Bonatti and R. Potrie,
	\newblock \emph{Many intermingled basins in dimension 3},
	\newblock \emph{Israel J. Math.}, \textbf{224} (2018), 293-314.
	
	\bibitem[4]{BV00} (MR1749677) [10.1007/BF02810585]
	\newblock Ch. Bonatti and M. Viana,
	\newblock \emph{S{RB} measures for partially hyperbolic systems whose central direction is mostly contracting},
	\newblock \emph{Israel J. Math.}, \textbf{115} (2000), 157-193.
	
	\bibitem[5]{BR75} (MR380889) [10.1007/BF01389848]
	\newblock R. Bowen and D. Ruelle,
	\newblock \emph{The ergodic theory of {A}xiom {A} flows},
	\newblock \emph{Invent. Math.}, \textbf{29} (1975), 181-202.
	
	\bibitem[6]{CMY22} (MR4405573) [10.1007/s00220-022-04338-5]
	\newblock Y. Cao, Z. Mi and D. Yang,
	\newblock \emph{On the abundance of {S}inai-{R}uelle-{B}owen measures},
	\newblock \emph{Comm. Math. Phys.}, \textbf{391} (2022), 1271-1306.
	
	\bibitem[7]{CO21} (MR4270268) [10.4310/MRL.2021.v28.n3.a2]
	\newblock P. D. Carrasco and D. Obata,
	\newblock \emph{A new example of robustly transitive diffeomorphism},
	\newblock \emph{Math. Res. Lett.}, \textbf{28} (2021), 665-679.
	
	\bibitem[8]{CGS18} (MR3721879) [10.3934/dcds.2018037]
	\newblock C. Cheng, S. Gan and Y. Shi,
	\newblock \emph{A robustly transitive diffeomorphism of {K}an's type},
	\newblock \emph{Discrete Contin. Dyn. Syst.}, \textbf{38} (2018), 867-888.
	
	\bibitem[9]{CYZ20} (MR4082180) [10.1007/s00220-019-03668-1]
	\newblock S. Crovisier, D. Yang and J. Zhang,
	\newblock \emph{Empirical measures of partially hyperbolic attractors},
	\newblock \emph{Comm. Math. Phys.}, \textbf{375} (2020), 725-764.
	
	\bibitem[10]{DVY16} (MR3452277) [10.1007/s00220-015-2554-y]
	\newblock D. Dolgopyat, M. Viana and J. Yang,
	\newblock \emph{Geometric and measure-theoretical structures of maps with mostly contracting center},
	\newblock \emph{Comm. Math. Phys.}, \textbf{341} (2016), 991-1014.
	
	\bibitem[11]{GS19} (MR3926097) [10.1016/j.jde.2018.11.029]
	\newblock S. Gan and Y. Shi,
	\newblock \emph{Robustly topological mixing of {K}an's map},
	\newblock \emph{J. Differential Equations}, \textbf{266} (2019), 7173-7196.
	
	\bibitem[12]{HYY20} (MR4042879) [10.1090/tran/7920]
	\newblock Y. Hua, F. Yang and J. Yang,
	\newblock \emph{A new criterion of physical measures for partially hyperbolic diffeomorphisms},
	\newblock \emph{Trans. Amer. Math. Soc.}, \textbf{373} (2020), 385-417.
	
	\bibitem[13]{IKS08} (MR2434457) [10.1007/s11784-008-0088-z]
	\newblock Yu. S. Ilyashenko, V. A. Kleptsyn and P. Saltykov,
	\newblock \emph{Openness of the set of boundary preserving maps of an annulus with intermingled attracting basins},
	\newblock \emph{J. Fixed Point Theory Appl.}, \textbf{3} (2008), 449-463.
	
	\bibitem[14]{Kan94} (MR1254075) [10.1090/S0273-0979-1994-00507-5]
	\newblock I. Kan,
	\newblock \emph{Open sets of diffeomorphisms having two attractors, each with an everywhere dense basin},
	\newblock \emph{Bull. Amer. Math. Soc.}, \textbf{31} (1994), 68-74.
	
	\bibitem[15]{MW05} (MR2183303) [10.1017/S0143385705000325]
	\newblock I. Melbourne and A. Windsor,
	\newblock \emph{A {$C^\infty$} diffeomorphism with infinitely many intermingled basins},
	\newblock \emph{Ergodic Theory Dynam. Systems}, \textbf{25} (2005), 1951-1959.
	
	\bibitem[16]{Na11} (MR2809110) [10.7208/chicago/9780226569505.001.0001]
	\newblock A. Navas,
	\newblock \emph{Groups of Circle Diffeomorphisms},
	\newblock Translation of the 2007 Spanish edition, University of Chicago Press, Chicago, 2011.
	
	\bibitem[17]{NBRV21} (MR4313711) [10.1088/1361-6544/ac1f79]
	\newblock B. N\'{u}\~{n}ez{-}Madariaga, S. A. Ram\'{\i}rez and C. H. V\'{a}squez,
	\newblock \emph{Measures maximizing the entropy for {K}an endomorphisms},
	\newblock \emph{Nonlinearity}, \textbf{34} (2021), 7255-7302.
	
	\bibitem[18]{Ok17} (MR3625005) [10.1007/s10883-016-9334-7]
	\newblock A. Okunev,
	\newblock \emph{Milnor attractors of skew products with the fiber a circle},
	\newblock \emph{J. Dyn. Control Syst.}, \textbf{23} (2017), 421-433.
	
	\bibitem[19]{PSW97} (MR1432307) [10.1215/S0012-7094-97-08616-6]
	\newblock C. Pugh, M. Shub and A. Wilkinson,
	\newblock \emph{H\"{o}lder foliations},
	\newblock \emph{Duke Math. J.}, \textbf{86} (1997), 517-546.
	
	\bibitem[20]{RT22} (MR4405658) [10.1007/s00209-021-02925-1]
	\newblock J. E. Rocha and A. Tahzibi,
	\newblock \emph{On the number of ergodic measures of maximal entropy for partially hyperbolic diffeomorphisms with compact center leaves},
	\newblock \emph{Math. Z.}, \textbf{301} (2022), 471-484.
	
	\bibitem[21]{Ru76} (MR415683) [10.2307/2373810]
	\newblock D. Ruelle,
	\newblock \emph{A measure associated with axiom-{A} attractors},
	\newblock \emph{Amer. J. Math.}, \textbf{98} (1976), 619-654.
	
	\bibitem[22]{Si72} (MR399421)
	\newblock Ja. G. Sina\u{\i},
	\newblock \emph{Gibbs measures in ergodic theory},
	\newblock \emph{Uspehi Mat. Nauk}, \textbf{27} (1972), 21-64.
	
	\bibitem[23]{UV18} (MR3742552) [10.1017/etds.2016.33]
	\newblock R. Ures and C. H. V\'{a}squez,
	\newblock \emph{On the non-robustness of intermingled basins},
	\newblock \emph{Ergodic Theory Dynam. Systems}, \textbf{38} (2018), 384-400.
			
\end{thebibliography}

\end{document}